\documentclass[12pt,a4paper]{amsart}
\usepackage[width=0.75\paperwidth,height=0.85\paperheight,twoside=false,includehead,includefoot]{geometry}
\allowdisplaybreaks[1]

\DeclareFontEncoding{OT2}{}{}
\DeclareFontSubstitution{OT2}{cmr}{m}{n}

\DeclareFontFamily{OT2}{cmr}{\hyphenchar\font45 }
\DeclareFontShape{OT2}{cmr}{m}{n}{%
   <5><6><7><8><9>gen*wncyr%
   <10><10.95><12><14.4><17.28><20.74><24.88>wncyr10}{}

\DeclareMathAlphabet{\mathcyr}{OT2}{cmr}{m}{n}
\DeclareMathAlphabet{\mathcyb}{OT2}{cmr}{b}{n}
\SetMathAlphabet{\mathcyr}{bold}{OT2}{cmr}{b}{n}

\usepackage{amsopn}

\usepackage{amstext}
\usepackage{amsthm}
\usepackage{amssymb}

\newtheorem{thm}{Theorem}[section]
\newtheorem{lem}[thm]{Lemma}
\newtheorem{prop}[thm]{Proposition}

\theoremstyle{definition}

\theoremstyle{remark}
\newtheorem{rem}[thm]{Remark}

\begin{document}

\title[{Connectors of the Ohno relation for PMZSs}]{Connectors of the Ohno relation for parametrized multiple zeta series}

\author{Hideki Murahara}
\address[Hideki Murahara]{Nakamura Gakuen University Graduate School, 5-7-1, Befu, Jonan-ku,
Fukuoka, 814-0198, Japan}
\email{hmurahara@nakamura-u.ac.jp}

\author{Tomokazu Onozuka}
\address[Tomokazu Onozuka]{Institute of Mathematics for Industry, Kyushu University 744, Motooka, Nishi-ku,
Fukuoka, 819-0395, Japan}
\email{t-onozuka@math.kyushu-u.ac.jp}


\begin{abstract}
 The Ohno relation is a well known relation in the theory of multiple zeta values. 
 Recently, Seki and Yamamoto introduced a connector method and gave its succinct proof. 
 On the other hand, Igarashi obtained the generalization of the Ohno relation for parametrized multiple zeta series.  
 In this paper, we give new connectors to simplify his proof.  
\end{abstract}

\maketitle

\section{Introduction}
Inspired by Seki and Yamamoto's recent works (see \cite{SY19}, \cite{SY20}, \cite{Sek20}, and \cite{Yam20}), we found new connectors of the Ohno relation for parametrized multiple zeta series (PMZSs), which was obtained by Igarashi in \cite{Iga12}. 
For positive integers $k_1,\ldots,k_r$ with $k_r\ge2$, the multiple zeta values (MZVs) are defined by
\begin{align*}
 \zeta(k_{1},\dots,k_{r}):=\sum_{1\le m_{1}<\cdots<m_{r}}\frac{1}{m_{1}^{k_{1}}\cdots m_{r}^{k_{r}}}.
\end{align*}

Let $m_0=n_{0}=0$, and $r$ and $s$ be positive integers. 
In \cite{SY19}, Seki and Yamamoto used the equalities
\begin{align}
  \label{PC}
  \frac{1}{m_r} \cdot \frac{m_{r}! n_{s-1}!}{(m_{r}+n_{s-1})!} 
  &=\sum_{n_{s-1}<n_s} \frac{1}{n_s} \cdot \frac{m_r! n_s!}{(m_r+n_s)!}, \\ 
  \label{CP}
  \sum_{m_{r-1}<m_r} \frac{1}{m_r} \cdot \frac{m_r! n_s!}{(m_r+n_s)!} 
  &=\frac{1}{n_s} \cdot \frac{m_{r-1}! n_s!}{(m_{r-1}+n_s)!} 
\end{align}
to prove the duality relation for MZVs. 
For example, we have
\begin{align*}
 \zeta(1,2) 
 &=\sum_{0<m_1<m_2} \frac{1}{m_1 m_2^2} \\
 &=\sum_{0<m_1<m_2} \sum_{0<n_1} \frac{1}{m_1 m_2} \cdot \frac{1}{n_1} \cdot \frac{m_2! n_1!}{(m_2+n_1)!}
  \qquad (\textrm{by } \eqref{PC} \textrm{ for } r=2,s=1) \\ 
 &=\sum_{0<m_1} \sum_{0<n_1} \frac{1}{m_1} \cdot \frac{1}{n_1^2} \cdot \frac{m_1! n_1!}{(m_1+n_1)!} 
  \qquad (\textrm{by } \eqref{CP} \textrm{ for } r=2,s=1) \\ 
 &=\sum_{0<n_1} \frac{1}{n_1^3} 
  \qquad (\textrm{by } \eqref{CP} \textrm{ for } r=1,s=1) \\ 
 &=\zeta(3).
\end{align*}
Note that the sum 
\[ 
 \sum_{0<m_1<\cdots<m_r} \sum_{0<n_1<\cdots<n_s} \frac{1}{m_1^{k_1} \cdots m_r^{k_r}} \cdot \frac{1}{n_1^{l_1} \cdots n_s^{l_s}} \cdot \frac{m_r! n_s!}{(m_r+n_s)!}
\] 
and 
\[ 
 \frac{m_r! n_s!}{(m_r+n_s)!}
\]
are called the connected sum and the connector of the duality relation, respectively. 

For positive integers $k_1,\ldots,k_r$ with $k_r\ge2$ and a complex number $\alpha$ with $\Re\alpha>0$, the PMZSs are defined by
\begin{align*}
 \zeta(k_{1},\dots,k_{r};\alpha)
 :=\sum_{0\le m_{1}<\cdots<m_{r}} 
  \frac{ (\alpha)_{m_1} }{ m_1! } \cdot
  \frac{ m_r! }{ (\alpha)_{m_r} } \cdot 
  \frac{1}{(m_{1}+\alpha)^{k_{1}}\cdots (m_{r}+\alpha)^{k_{r}}},
\end{align*}
where 
\begin{align*}
 (\alpha)_{m}:=
 \begin{cases}
  \alpha(\alpha+1)\cdots(\alpha+m-1) 
   &\textrm{if } m\in\mathbb{Z}_{\ge1}, \\
  1
   &\textrm{if } m=0.
 \end{cases}
\end{align*}
Note that this series is a Hurwitz type generalization of MZVs, which is know to satisfy the  cyclic sum  formula \cite{Iga11} and the Ohno relation \cite{Iga12}.

We call a tuple of positive integers $(k_1,\dots,k_r)$ with $k_r\ge2$ admissible index. 
For a positive integer $a$ and a non-negative integer $m$, let $(\{a\}^m):=(\underbrace{a,\dots,a}_{m})$.
If we write an admissible index as
 \[
  (\{1\}^{a_1-1},b_1+1,\dots,\{1\}^{a_t-1},b_t+1) \quad (a_p, b_q\ge1),
 \]
we define its dual index as 
 \[
  (\{1\}^{b_t-1},a_t+1,\dots,\{1\}^{b_1-1},a_1+1).
 \]
Then the duality relation of MZVs is described as 
\[
 \zeta (k_1,\dots,k_{r})
 =\zeta (k'_1,\dots,k'_{r'}), 
\]
where $(k'_1,\dots,k'_{r'})$ is the dual index of an admissible index $(k_1,\dots,k_r)$. 
Similarly, the duality relation of PMZSs can be written as 
\[
 \zeta (k_1,\dots,k_{r};\alpha)
 =\zeta (k'_1,\dots,k'_{r'};\alpha), 
\]
which is a special case of the following theorem. 
\begin{thm}[Ohno relation for PMZSs; Igarashi \cite{Iga12}] \label{main}
 Let $(k_1,\dots,k_r)$ be an admissible index and $(k'_1,\dots,k'_{r'})$ be its dual index. 
  For a non-negative integer $m$ and a complex number $\alpha$ with $\Re\alpha>0$, we have
 \begin{align*}
  \sum_{\substack{ e_1+\cdots+e_r=m \\ e_1,\dots,e_r\ge0 }} 
  \zeta(k_1+e_1,\dots,k_r+e_r;\alpha)
  =\sum_{\substack{ e'_1+\cdots+e'_{r'}=m \\ e'_1,\dots,e'_{r'}\ge0 }} 
  \zeta(k'_1+e'_1,\dots,k'_{r'}+e'_{r'};\alpha).
 \end{align*}
\end{thm}
\begin{rem}
 This is a generalization of the result obtained by Ohno in \cite{Ohn99}. 
 Seki and Yamamoto found the connector of his result i.e., the Ohno relation for MZVs in \cite{SY20}.  
\end{rem}

Now we reset $m_0:=-1$ and $n_0:=-1$. 
In this paper, we give an alternative proof of Theorem \ref{main} by using the following relations (for the proofs, see next section):
\begin{align}
 \begin{split} \label{PCO1}
  &\frac{1}{ m_r+\alpha } 
  \\
  &=\frac{\Gamma(\alpha)}{\Gamma(\alpha-x)}\sum_{0\le n_s} 
    \frac{ 1}{ n_s+\alpha-x } \cdot 
    \frac{ (\alpha)_{m_r} }{m_r! } \cdot
   \frac{ (\alpha)_{n_s} }{ n_s! } \cdot
   \frac{\Gamma(m_r+\alpha-x+1) \Gamma(n_s+\alpha-x+1)}{\Gamma(m_r+n_s+2\alpha-x+1)},
 \end{split} \\
 \begin{split} \label{CPO1}
  &\frac{\Gamma(\alpha)}{\Gamma(\alpha-x)}\sum_{0\le m_r} 
   \frac{ 1 }{ m_r+\alpha-x } \cdot 
   \frac{ (\alpha)_{m_r} }{ m_r! } \cdot 
   \frac{ (\alpha)_{n_s} }{ n_s! }\cdot 
   \frac{\Gamma(m_r+\alpha-x+1) \Gamma(n_s+\alpha-x+1)}{\Gamma(m_r+n_s+2\alpha-x+1)} \\
  &=\frac{1}{ n_s+\alpha }    
 \end{split}
\end{align}
and
\begin{align}
  \label{PCO2}
  \begin{split}
  &\frac{1}{m_r+\alpha} \cdot 
  \frac{\Gamma(m_{r}+\alpha-x+1) \Gamma(n_{s-1}+\alpha-x+1) }{ \Gamma(m_{r}+n_{s-1}+2\alpha-x+1) } \\
  &=\sum_{n_{s-1}<n_s} \frac{1}{n_s+\alpha-x} \cdot 
   \frac{ \Gamma(m_r+\alpha-x+1) \Gamma(n_s+\alpha-x+1) }{ \Gamma(m_r+n_s+2\alpha-x+1) }, 
  \end{split} \\ 
  \label{CPO2}
  \begin{split}
  &\sum_{m_{r-1}<m_r} 
   \frac{1}{m_r+\alpha-x} \cdot 
   \frac{ \Gamma(m_r+\alpha-x+1) \Gamma(n_s+\alpha-x+1) }{ \Gamma(m_r+n_s+2\alpha-x+1) } \\
  &=\frac{1}{n_s+\alpha} \cdot 
   \frac{ \Gamma(m_{r-1}+\alpha-x+1) \Gamma(n_s+\alpha-x+1) }{ \Gamma(m_{r-1}+n_s+2\alpha-x+1) },
  \end{split}
\end{align}
where $x$ is a formal parameter. 
These two types of relations are analogues of \eqref{PC} and \eqref{CP}, 
 and we apply these relations repeatedly to prove Theorem \ref{main}.
More precisely, we use `\eqref{PCO1} and \eqref{CPO1}' in the first and last steps, and `\eqref{PCO2} and \eqref{CPO2}' in the middle steps.
Note that our connectors are
\[
 \frac{\Gamma(\alpha)}{\Gamma(\alpha-x)} \cdot    
 \frac{ (\alpha)_{m_r} }{ m_r! } \cdot 
 \frac{ (\alpha)_{n_s} }{ n_s! }\cdot 
 \frac{\Gamma(m_r+\alpha-x+1) \Gamma(n_s+\alpha-x+1)}{\Gamma(m_r+n_s+2\alpha-x+1)}
\] 
and
\[
 \frac{ \Gamma(m_r+\alpha-x+1) \Gamma(n_s+\alpha-x+1) }{ \Gamma(m_r+n_s+2\alpha-x+1) } , 
\]
respectively. 
In general, finding a connector might be difficult, but once it is obtained, proofs often become easier.
Especially, to prove the duality relation for PMZSs, we use the following relations: 
\begin{align}
  \label{PC1}
  \frac{1}{m_r+\alpha} 
  =\sum_{0\le n_s} 
   \frac{ 1}{n_s+\alpha} \cdot 
   \frac{ (\alpha)_{m_r} }{ m_r! } \cdot
   \frac{ (\alpha)_{n_s} }{ n_s! } \cdot 
   \frac{ \Gamma(m_r+\alpha+1) \Gamma(n_s+\alpha+1)}{ \Gamma(m_r+n_s+2\alpha+1) }, \\
  \label{CP1}
  \sum_{0\le m_r} 
   \frac{ 1}{m_r+\alpha} \cdot 
   \frac{ (\alpha)_{m_r} }{ m_r! } \cdot
   \frac{ (\alpha)_{n_s} }{ n_s! } \cdot 
   \frac{ \Gamma(m_r+\alpha+1) \Gamma(n_s+\alpha+1)}{ \Gamma(m_r+n_s+2\alpha+1) } 
  =\frac{1}{n_s+\alpha} 
\end{align}
and
\begin{align}
  \label{PC2}
  \begin{split}
  &\frac{1}{m_r+\alpha} \cdot 
  \frac{ \Gamma(m_{r}+\alpha+1) \Gamma(n_{s-1}+\alpha+1) }{ \Gamma(m_{r}+n_{s-1}+2\alpha+1) } \\
  &=\sum_{n_{s-1}<n_s} 
  \frac{1}{n_s+\alpha} \cdot 
  \frac{ \Gamma(m_r+\alpha+1) \Gamma(n_s+\alpha+1) }{ \Gamma(m_r+n_s+2\alpha+1) }, 
  \end{split} \\ 
  \label{CP2}
  \begin{split}
  &\sum_{m_{r-1}<m_r} 
   \frac{1}{m_r+\alpha} \cdot 
   \frac{ \Gamma(m_r+\alpha+1) \Gamma(n_s+\alpha+1) }{ (m_r+n_s+2\alpha+1) } \\
  &=\frac{1}{n_s+\alpha} \cdot 
   \frac{ \Gamma(m_{r-1}+\alpha+1) \Gamma(n_s+\alpha+1) }{ \Gamma(m_{r-1}+n_s+2\alpha+1) }. 
  \end{split}
\end{align}
%
For example, we have
\begin{align*}
 &\zeta(2,3;\alpha) \\
 &=\sum_{0\le m_1<m_2} 
  \frac{ (\alpha)_{m_1} }{ m_1! } \cdot
  \frac{ m_2! }{ (\alpha)_{m_2} } \cdot
  \frac{1}{(m_1+\alpha)^2 (m_2+\alpha)^3} \\
 &=\sum_{0\le m_1<m_2} 
  \sum_{0\le n_1} 
  \frac{ (\alpha)_{m_1} }{ m_1! } \cdot
  \frac{ (\alpha)_{n_1} }{ n_1! } \cdot 
  \frac{1}{(m_1+\alpha)^2 (m_2+\alpha)^2} \cdot 
  \frac{1}{n_1+\alpha} \\ 
  &\qquad\qquad\qquad\qquad\qquad 
   \cdot \frac{ \Gamma(m_2+\alpha+1) \Gamma(n_1+\alpha+1)}{ \Gamma(m_2+n_1+2\alpha+1) } 
   \qquad
   (\textrm{by } \eqref{PC1} \textrm{ for } r=2,s=1) \\ 
 &=\sum_{0\le m_1<m_2} 
  \sum_{0\le n_1<n_2} 
  \frac{ (\alpha)_{m_1} }{ m_1! } \cdot
  \frac{ (\alpha)_{n_1} }{ n_1! } \cdot 
  \frac{1}{ (m_1+\alpha)^2 (m_2+\alpha) } \cdot 
  \frac{1}{ (n_1+\alpha) (n_2+\alpha) } \\
  &\qquad\qquad\qquad\qquad\qquad 
   \cdot \frac{ \Gamma(m_2+\alpha+1) \Gamma(n_2+\alpha+1)}{ \Gamma(m_2+n_2+2\alpha+1) } 
   \qquad
   (\textrm{by } \eqref{PC2} \textrm{ for } r=2,s=2) \\ 
 &=\sum_{0\le m_1} 
  \sum_{0\le n_1<n_2} 
  \frac{ (\alpha)_{m_1} }{ m_1! } \cdot
  \frac{ (\alpha)_{n_1} }{ n_1! } \cdot 
  \frac{1}{ (m_1+\alpha)^2 } \cdot 
  \frac{1}{ (n_1+\alpha) (n_2+\alpha)^2 } \\
  &\qquad\qquad\qquad\qquad\qquad 
   \cdot \frac{ \Gamma(m_1+\alpha+1) \Gamma(n_2+\alpha+1)}{ \Gamma(m_1+n_2+2\alpha+1) } 
   \qquad
   (\textrm{by } \eqref{CP2} \textrm{ for } r=2,s=2) \\ 
 &=\sum_{0\le m_1} 
  \sum_{0\le n_1<n_2<n_3} 
  \frac{ (\alpha)_{m_1} }{ m_1! } \cdot
  \frac{ (\alpha)_{n_1} }{ n_1! } \cdot 
  \frac{1}{ (m_1+\alpha) } \cdot 
  \frac{1}{ (n_1+\alpha) (n_2+\alpha)^2 (n_3+\alpha) } \\
  &\qquad\qquad\qquad\qquad\qquad 
   \cdot \frac{ \Gamma(m_1+\alpha+1) \Gamma(n_3+\alpha+1)}{ \Gamma(m_1+n_3+2\alpha+1) } 
   \qquad
   (\textrm{by } \eqref{PC2} \textrm{ for } r=1,s=3) \\ 
 &=\sum_{0\le n_1<n_2<n_3} 
  \frac{ (\alpha)_{n_1} }{ n_1! } \cdot 
  \frac{ n_3! }{ (\alpha)_{n_3} } \cdot
  \frac{1}{ (n_1+\alpha) (n_2+\alpha)^2 (n_3+\alpha)^2 } 
  \qquad 
   (\textrm{by } \eqref{CP1} \textrm{ for } r=1,s=3) \\ 
 &=\zeta(1,2,2;\alpha).
\end{align*}

As mentioned before, we used two types of relations \eqref{PCO1} to \eqref{CPO2} (or equivalently \eqref{PC1} to \eqref{CP2}) in the above equalities. However, we can use only one type of relations \eqref{PCO2} and \eqref{CPO2} (or \eqref{PC2} and \eqref{CP2}) to show the following relation:

\begin{prop}\label{main2}
 Let $(k_1,\dots,k_r)$ be an admissible index and $(k'_1,\dots,k'_{r'})$ be its dual index, and $\alpha$ be a complex number with $\Re\alpha>0$. 
 Then we have
 \[
 \widetilde\zeta(k_{1},\dots,k_{r};\alpha)= \widetilde\zeta(k'_{1},\dots,k'_{r'};\alpha),
 \]
 where
\begin{align*}
 \widetilde\zeta(k_{1},\dots,k_{r};\alpha)
 :=\sum_{0\le m_{1}<\cdots<m_{r}} 
  \frac{ \Gamma(m_r+\alpha+1) }{ \Gamma(m_r+2\alpha) }\cdot 
  \frac{1}{(m_{1}+\alpha)^{k_{1}}\cdots (m_{r}+\alpha)^{k_{r}}}.
\end{align*}
\end{prop}
\begin{rem}
Generally, we can obtain \eqref{sec3}, which is a generalization of the above proposition with a formal parameter $x$.
\end{rem}
From Theorem \ref{main} and Proposition \ref{main2}, there might be another factor $F_{(m_1,\ldots,m_r)}(\alpha)$ satisfying  the duality relation if we consider the series
\begin{align*}
 \xi(k_{1},\dots,k_{r};\alpha)
 :=\sum_{0\le m_{1}<\cdots<m_{r}} 
  F_{(m_1,\ldots,m_r)}(\alpha)\cdot 
  \frac{1}{(m_{1}+\alpha)^{k_{1}}\cdots (m_{r}+\alpha)^{k_{r}}}.
\end{align*}

\section{Proofs}
\subsection{Proof of Theorem \ref{main}} 
For positive integers $k_1,\dots,k_r,l_1,\dots,l_s$ and a complex number $\alpha$ with $\Re\alpha>0$, let
\begin{align*}
 &Z(k_1,\dots,k_r; \emptyset; \alpha; x) \\
 &:=\sum_{0\le m_1<\cdots<m_r} 
  \frac{ (\alpha)_{m_1} }{ m_1! } \cdot
  \frac{ m_r! }{ (\alpha)_{m_r} } \cdot
  \frac{1}{(m_1+\alpha)^{k_1-1}(m_1+\alpha-x) \cdots (m_r+\alpha)^{k_r-1}(m_r+\alpha-x) } \\
  &\qquad\qquad\qquad\qquad\qquad\qquad\qquad\qquad\qquad\qquad\qquad\qquad\qquad\qquad\qquad\qquad\qquad(k_r\ge2),\\
 &Z(\emptyset; l_1,\dots,l_s; \alpha; x) \\
 &:=\sum_{0\le n_1<\cdots<n_s} 
  \frac{ (\alpha)_{n_1} }{ n_1! } \cdot
  \frac{ n_s! }{ (\alpha)_{n_s} } \cdot
  \frac{1}{(n_1+\alpha)^{l_1-1}(n_1+\alpha-x) \cdots (n_s+\alpha)^{l_s-1}(n_s+\alpha-x) } \\
  &\qquad\qquad\qquad\qquad\qquad\qquad\qquad\qquad\qquad\qquad\qquad\qquad\qquad\qquad\qquad\qquad\qquad(l_s\ge2),\\
 &Z(k_1,\dots,k_r; l_1,\dots,l_s; \alpha; x) \\
 &:=\frac{\Gamma(\alpha)}{\Gamma(\alpha-x)} 
  \sum_{0\le m_1<\cdots<m_r} 
  \sum_{0\le n_1<\cdots<n_s} 
  \frac{ (\alpha)_{m_1} }{ m_1! } \cdot
  \frac{ (\alpha)_{n_1} }{ n_1! } \cdot 
  \frac{ \Gamma(m_r+\alpha-x+1) \Gamma(n_s+\alpha-x+1)}{ \Gamma(m_r+n_s+2\alpha-x+1) }  \\
  &\qquad \qquad \qquad\qquad\qquad\qquad 
   \cdot \frac{1}{(m_1+\alpha)^{k_1-1}(m_1+\alpha-x) \cdots (m_r+\alpha)^{k_r-1}(m_r+\alpha-x) } \\ 
  &\qquad \qquad \qquad\qquad\qquad\qquad 
   \cdot \frac{1}{(n_1+\alpha)^{l_1-1}(n_1+\alpha-x) \cdots (n_s+\alpha)^{l_s-1}(n_s+\alpha-x) }.
\end{align*}
Note that $Z(k_1,\dots,k_r; \emptyset; \alpha; x)=Z(\emptyset; k'_1,\dots,k'_{r'}; \alpha; x)$ states the Ohno relation for PMZSs (Theorem \ref{main}) if the index $(k'_1,\dots,k'_{r'})$ is the dual index of an admissible index $(k_1,\dots,k_r)$. 

To prove Theorem \ref{main}, we need Lemmas \ref{lem1} and \ref{lem2}. 
\begin{lem} \label{lem1}
 For a non-negative integer $m$ and a complex number $\alpha$ with $\Re\alpha>0$, we have
 \begin{align*}
  \sum_{0\le n} 
  \frac{1}{n+\alpha-x} \cdot
  \frac{ (\alpha)_{n} }{ n! } \cdot
  \frac{ \Gamma(m+\alpha-x+1)\Gamma(n+\alpha-x+1) }{ \Gamma(m+n+2\alpha-x+1) }
  =\frac{\Gamma(\alpha-x)}{\Gamma(\alpha)} \cdot\frac{1}{m+\alpha}\cdot\frac{  m! }{ (\alpha)_m }. 
 \end{align*}
\end{lem}
\begin{proof}
 Since
 \begin{align*}
  &\sum_{0\le n} 
  \frac{1}{n+\alpha-x} \cdot
  \frac{ \Gamma(m+\alpha-x+1)\Gamma(n+\alpha-x+1) }{ \Gamma(m+n+2\alpha-x+1) } \cdot
  \frac{ (\alpha)_{n} }{ n! } y^{n} \\
  &=\frac{ \Gamma(m+\alpha-x+1)\Gamma(\alpha-x) }{ \Gamma(m+2\alpha-x+1) } 
  \cdot {_2F_1}(\alpha-x, \alpha, m+2\alpha-x+1; y), 
 \end{align*}
 and Gauss's hypergeometric theorem
 \begin{align*}
   {_2F_1}(a, b, c; 1)
   =\frac{ \Gamma(c) \Gamma(c-a-b) }{ \Gamma(c-a) \Gamma(c-b) },
 \end{align*} 
 we have the result.  
\end{proof}

\begin{lem} \label{lem2}
 For a positive integer $m$ and a complex number $\alpha$ with $\Re\alpha>0$, we have
\begin{align*}
 &\sum_{n_{s-1}<n_s} \frac{1}{n_s+\alpha-x} \cdot 
  \frac{ \Gamma(m_r+\alpha-x+1) \Gamma(n_s+\alpha-x+1) }{ \Gamma(m_r+n_s+2\alpha-x+1) } \\
 &=\frac{1}{m_r+\alpha} \cdot 
  \frac{\Gamma(m_{r}+\alpha-x+1) \Gamma(n_{s-1}+\alpha-x+1) }{ \Gamma(m_{r}+n_{s-1}+2\alpha-x+1) }.
\end{align*}
\end{lem}
\begin{proof}
Since
\[
 \Gamma(n_s+\alpha-x)=\frac{ \Gamma(n_s+\alpha-x+1) }{ n_s+\alpha-x }
 \quad\textrm{ and }\quad
 \frac{1}{ \Gamma(m_r+n_s+2\alpha-x) }=\frac{ m_r+n_s+2\alpha-x }{ \Gamma(m_r+n_s+2\alpha-x+1) },
\]
we have
\begin{align*}
 &\sum_{n_{s-1}<n_s} \frac{1}{n_s+\alpha-x} \cdot 
  \frac{ \Gamma(m_r+\alpha-x+1) \Gamma(n_s+\alpha-x+1) }{ \Gamma(m_r+n_s+2\alpha-x+1) } \\
 &=\frac{1}{m_r+\alpha} 
  \sum_{n_{s-1}<n_s}  
  \left(
   \frac{ \Gamma(m_r+\alpha-x+1) \Gamma(n_s+\alpha-x) }{ \Gamma(m_r+n_s+2\alpha-x) } 
   -\frac{ \Gamma(m_r+\alpha-x+1) \Gamma(n_s+\alpha-x+1) }{ \Gamma(m_r+n_s+2\alpha-x+1) } 
  \right) 
  \\
 &=\frac{1}{m_r+\alpha} \cdot 
  \frac{\Gamma(m_{r}+\alpha-x+1) \Gamma(n_{s-1}+\alpha-x+1) }{ \Gamma(m_{r}+n_{s-1}+2\alpha-x+1) }.
\end{align*}
This finishes the proof. 
\end{proof}

\begin{proof}[Proof of Theorem \ref{main}]
 By Lemma \ref{lem1}, we have
 \[
  Z(k_1,\dots,k_r; \emptyset; \alpha; x)
  =Z(k_1,\dots,k_r-1; 1; \alpha; x)
 \]
 for positive integers $k_1,\dots,k_r$ with $k_r\ge2$,
 and 
 \[
  Z(\emptyset; l_1,\dots,l_s; \alpha; x)
  =Z(1; l_1,\dots,l_s-1; \alpha; x)
 \]
 for positive integers $l_1,\dots,l_s$ with $l_s\ge2$. 
  By Lemma \ref{lem2}, we also have
 \[
  Z(k_1,\dots,k_r+1; l_1,\dots,l_s; \alpha; x)
  =Z(k_1,\dots,k_r; l_1,\dots,l_s,1; \alpha; x)
 \]
 and 
 \[
  Z(k_1,\dots,k_r,1; l_1,\dots,l_s; \alpha; x)
  =Z(k_1,\dots,k_r; l_1,\dots,l_s+1; \alpha; x)
 \]
 for positive integers $k_1,\dots,k_r, l_1,\dots,l_s$. 
 Thus we obtain the result. 
\end{proof}


\subsection{Proof of Proposition \ref{main2}}
For positive integers $k_1,\dots,k_r,l_1,\dots,l_s$ and a complex number $\alpha$ with $\Re\alpha>0$, let
\begin{align*}
 &\widetilde Z(k_1,\dots,k_r; \emptyset; \alpha; x) \\
 &:=\sum_{0\le m_1<\cdots<m_r} 
  \frac{ \Gamma(m_r+\alpha-x+1)  }{ \Gamma(m_r+2\alpha-x) }\cdot
  \frac{1}{(m_1+\alpha)^{k_1-1}(m_1+\alpha-x) \cdots (m_r+\alpha)^{k_r-1}(m_r+\alpha-x) } \\
   &\qquad\qquad\qquad\qquad\qquad\qquad\qquad\qquad\qquad\qquad\qquad\qquad\qquad\qquad\qquad\qquad\qquad(k_r\ge2),\\
 &\widetilde Z(\emptyset; l_1,\dots,l_s; \alpha; x) \\
 &:=\sum_{0\le n_1<\cdots<n_s} 
  \frac{ \Gamma(n_s+\alpha-x+1)  }{ \Gamma(n_s+2\alpha-x) }\cdot
  \frac{1}{(n_1+\alpha)^{l_1-1}(n_1+\alpha-x) \cdots (n_s+\alpha)^{l_s-1}(n_s+\alpha-x) } \\
   &\qquad\qquad\qquad\qquad\qquad\qquad\qquad\qquad\qquad\qquad\qquad\qquad\qquad\qquad\qquad\qquad\qquad(l_s\ge2),\\
 &\widetilde Z(k_1,\dots,k_r; l_1,\dots,l_s; \alpha; x) \\
 &:=
  \frac{1}{\Gamma(\alpha-x)} \sum_{0\le m_1<\cdots<m_r} 
  \sum_{0\le n_1<\cdots<n_s} 
  \frac{ \Gamma(m_r+\alpha-x+1) \Gamma(n_s+\alpha-x+1)}{ \Gamma(m_r+n_s+2\alpha-x+1) }  \\
  &\qquad \qquad \qquad\qquad\qquad\qquad 
   \cdot \frac{1}{(m_1+\alpha)^{k_1-1}(m_1+\alpha-x) \cdots (m_r+\alpha)^{k_r-1}(m_r+\alpha-x) } \\ 
  &\qquad \qquad \qquad\qquad\qquad\qquad 
   \cdot \frac{1}{(n_1+\alpha)^{l_1-1}(n_1+\alpha-x) \cdots (n_s+\alpha)^{l_s-1}(n_s+\alpha-x) }.
\end{align*}

Similar to the proof of Theorem \ref{main}, by applying Lemma \ref{lem2} for all steps, we have 
\begin{align}\label{sec3}
\widetilde Z(k_1,\dots,k_r; \emptyset; \alpha; x)=\widetilde Z(\emptyset; k'_1,\dots,k'_{r'}; \alpha; x),
\end{align}
 where the index $(k'_1,\dots,k'_{r'})$ is the dual index of an admissible index $(k_1,\dots,k_r)$. Since $\widetilde Z(k_1,\dots,k_r; \emptyset; \alpha; 0)=\widetilde\zeta(k_{1},\dots,k_{r};\alpha)$, we have the result.

\section*{Acknowledgement}
This work was supported by JSPS KAKENHI Grant Number JP19K14511.


\end{document}